\documentclass[a4paper,11pt]{amsart}

\usepackage{amsmath}
\usepackage{amssymb}
\usepackage{amsthm}
\usepackage{verbatim}
\usepackage{enumerate}

\usepackage{color}

\usepackage{xcolor}

\newtheorem{thm}{Theorem}[section]
\newtheorem{prop}[thm]{Proposition}
\newtheorem{cor}[thm]{Corollary}
\newtheorem{lem}[thm]{Lemma}

\theoremstyle{definition}

\newtheorem{dfn}[thm]{Definition}

\newtheorem{rmk}[thm]{Remark}

\numberwithin{equation}{section}
\newcommand{\cS}{\mathcal{S}}

\newcommand{\boldA}{\mathbf{A}}

\newcommand{\Cn}{C^{n-1}(\mathbb{R}) \cap C^{n}(\mathbb{R} \backslash \{ 0 \})}

\title[Weak type estimates for multiple operator integrals of divided differences]{Weak $(1,1)$ estimates for multiple operator integrals and generalized absolute value functions}

\date{\noindent \today.  MSC2010: 47B10, 47L20, 47H60.  MC is supported by the NWO Vidi grant `Non-commutative harmonic analysis and rigidity of operator algebras', VI.Vidi.192.018. FS is supported by the ARC Laureate Fellowship.}

\author[Martijn Caspers]{Martijn Caspers}
\address{TU Delft, EWI/DIAM,
	P.O.Box 5031,
	2600 GA Delft,
	The Netherlands}
\email{m.p.t.caspers@tudelft.nl}

\author[Fedor Sukochev]{Fedor Sukochev}

\author[Dmitriy Zanin]{Dmitriy Zanin}
\address{School of Mathematics and Statistics, UNSW, Kensington 2052, NSW, Australia}
\email{f.sukochev@unsw.edu.au}
\email{d.zanin@unsw.edu.au}

\begin{document}

\maketitle

\begin{abstract}
Consider the generalized absolute value function defined by
\[
a(t) = \vert t \vert t^{n-1}, \qquad t \in \mathbb{R}, n \in \mathbb{N}_{\geq 1}.
\]
Further, consider the $n$-th order divided difference function $a^{[n]}: \mathbb{R}^{n+1} \rightarrow \mathbb{C}$ and let $1 < p_1, \ldots, p_n < \infty$ be such that $\sum_{l=1}^n  p_l^{-1} = 1$. Let $\cS_{p_l}$ denote the Schatten-von Neumann  ideals and let $\cS_{1,\infty}$ denote the weak  trace class ideal.  We show that for any $(n+1)$-tuple $\boldA$ of bounded self-adjoint operators  the multiple operator integral $T_{a^{[n]}}^{\boldA}$ maps
$\cS_{p_1} \times \ldots \times \cS_{p_n}$ to $\cS_{1, \infty}$
boundedly with uniform bound in $\boldA$.  The same is true for the class of $C^{n+1}$-functions that outside the interval $[-1, 1]$ equal $a$. In \cite{CLPST} it was proved that for a function $f$ in this class such boundedness of $T^{\boldA}_{f^{[n]}}$ from $\cS_{p_1} \times \ldots \times \cS_{p_n}$ to $\cS_{1}$ may fail, resolving a problem by V. Peller. This   shows that the estimates in the current paper  are optimal.
The proof is based on a new reduction method for arbitrary multiple operator integrals of divided differences.

\end{abstract}

\section{Introduction}

This paper is concerned with the following problem. Consider a Borel function $f: \mathbb{R} \rightarrow \mathbb{C}$. Consider  the divided difference function $f^{[n]}: \mathbb{R}^{n+1} \rightarrow \mathbb{C}$ and assume it is bounded. For an $(n+1)$-tuple $\boldA = (A_0, \ldots, A_n)$ of bounded self-adjoint operators, consider the multiple operator integral
\begin{equation}\label{Eqn=S2}
T^{\boldA}_{f^{[n]}}: \cS_2 \times \ldots \times \cS_2 \rightarrow \cS_2.
\end{equation}
Here $\cS_2$ is the   Hilbert-Schmidt ideal and by \cite{CLPST} the map \eqref{Eqn=S2} is well-defined. We now ask for an extension of the multi-linear map \eqref{Eqn=S2} to  other Schatten $\cS_p$-spaces. Such extensions have several important applications to  differentiability properties of functions on non-commutative spaces, see e.g. \cite{PSS}, \cite{SkripkaTomskova}.

\vspace{0.3cm}

\noindent {\bf Problem 1.} For which class of functions $f: \mathbb{R} \rightarrow \mathbb{C}$ and which values of $1 \leq  p_1, \ldots, p_n < \infty$ with $\sum_{l=1}^n \frac{1}{p_l} = 1$ does $T^{\boldA}_{f^{[n]}}$ for every $\boldA \in B(H)_{sa}^{\times n+1}$ extend to a bounded map
\begin{equation}\label{Eqn=S1Intro}
T^{\boldA}_{f^{[n]}}: \cS_{p_1} \times \ldots \times \cS_{p_n} \rightarrow \cS_{1, \infty}.
\end{equation}



\vspace{0.3cm}

In case $n = 1$ a complete solution to this problem was found by D. Potapov and the authors in \cite{CPSZ-AJM}. In this case   \eqref{Eqn=S1Intro} concerns boundeness of
\begin{equation}\label{Eqn=IntroMore}
T^{\boldA}_{f^{[1]}}: \cS_{1}   \rightarrow \cS_{1, \infty}.
\end{equation}
The main result of \cite{CPSZ-AJM} yields that \eqref{Eqn=IntroMore} is bounded uniformly in $\boldA  \in B(H)_{sa}^{\times 2}$ if and only if $f$ is Lipschitz. Preliminary results on this problem have been obtained by Nazarov and Peller \cite{NazarovPeller} for rank 1 operators and by the authors \cite{CPSZ-JOT} for $f$ the absolute value map. Through interpolation \cite{CPSZ-AJM} (see also \cite{CSZ})  it implies the main results of \cite{PSActa} and \cite{CMPS} as well as many previous results on perturbation of commutators and non-commutative Lipschitz properties. In this sense the so-called weak type $(1,1)$ estimate  \eqref{Eqn=IntroMore} is the optimal one. Crucial in the proof of \cite{CPSZ-AJM} is the connection to non-commutative Calder\'on-Zygmund theory and the results by Parcet \cite{Parcet} and Cadilhac \cite{Cadilhac}.

\vspace{0.3cm}

That Problem 1 is the right question to pose is further witnessed by the fact that there is no   uniform bound in $\boldA \in B(H)_{sa}^{n+1}$ of the map
\begin{equation}\label{Eqn=IntroS1Fail}
T^{\boldA}_{f^{[n]}}: \cS_{p_1} \times \ldots \times \cS_{p_n} \rightarrow \cS_{1}.
\end{equation}
  For $n = 1$ counterexamples were (in different but related contexts) obtained by  Farforovskaya \cite{Far6}, \cite{Far7}, \cite{Far8}, Kato \cite{Kato} and Davies \cite{Davies}. Most notably Davies proves in \cite{Davies} that the estimate \eqref{Eqn=IntroS1Fail} fails for $n=1$ and  for the absolute value map $f.$ For $n \geq 2$ negative results were obtained much more recently in \cite{CLPST}. The functions  that are used in \cite{CLPST} to show failure of a uniform bound in $\boldA \in B(H)_{sa}^{n+1}$  of \eqref{Eqn=IntroS1Fail} are variations of a generalized (higher order) absolute value map
\begin{equation}\label{Eqn=AIntro}
a(t) = \vert t \vert t^{n-1}, \qquad t \in \mathbb{R}.
\end{equation}
This class of functions is exactly the object of study of the current paper, see the final Remark \ref{Rmk=Final}. Further negative results for $n \geq 2$ can be found in \cite{PSAdvances}.

\vspace{0.3cm}

The results so far naturally motivate a study of Problem 1 for $n \geq 2$. Moreover, affirmative answers to Problem 1 for classes of functions give optimal solutions to some of the main results  in \cite{PSS}
where it was proved that for any $f \in C^n(\mathbb{R})$ with bounded $n$-th order derivative $f^{(n)}$ we have
\[
T_{f^{[n]}}^{\boldA}: \cS_{p_1} \times \ldots \times \cS_{p_n} \rightarrow \cS_{p}, \qquad 1 < p, p_1, \ldots, p_n < \infty, \textrm{ with } \sum_{l=1}^n \frac{1}{p_l}= \frac{1}{p}.
\]
Despite its importance up until now for $n \geq 2$   nothing is known about the boundedness of \eqref{Eqn=S1Intro} for any class of functions $f$ unless already the stronger estimate \eqref{Eqn=IntroS1Fail} holds.  This paper is the first attempt to fill in a void in this area. Namely, we give an affirmative answer to Problem 1 for the generalized absolute value function $a$ as well as for a class of   related functions. Note that these examples are historically the most natural ones, since several results have been obtained in the past for absolute value maps. In particular we show that the class of counterexamples obtained in \cite{CLPST} to the estimate \eqref{Eqn=IntroS1Fail} does satisfy the weak $(1,1)$ estimate \eqref{Eqn=S1Intro}.  For other results on absolute value maps in this context we refer  to  \cite{McIntosh}, \cite{Kato}, \cite{Davies}, \cite{Kosaki}, \cite{DDPS1}, \cite{DDPS2}, \cite{Akooi},  \cite{CPSZ-JOT}, \cite{CLPST}, \cite{PSAdvances}.

\vspace{0.3cm}

 Here is our main theorem. We draw the reader's attention that our assumption on the indices $p_1,\cdots,p_n$ below are wider than those in \cite{PSS}.  This explains a serious difference between our method and that of \cite{PSS}.

\begin{thm}\label{Thm=MainIntro}
Let $n \in \mathbb{N}_{\geq 1}$ and  $1 \leq  p_1, \ldots, p_n < \infty$ with $\sum_{l=1}^n \frac{1}{p_l} = 1$. There exists a constant $D(p_1, \ldots, p_n) > 0$ such that for every $\boldA \in B(H)_{sa}^{\times n+1}$ we have:
\begin{equation}\label{Eqn=S1Intro2}
\Vert T^{\boldA}_{f^{[n]}}: \cS_{p_1} \times \ldots \times \cS_{p_n} \rightarrow \cS_{1, \infty} \Vert \leq D(p_1, \ldots, p_n),
\end{equation}
where $f = a$ as defined in \eqref{Eqn=AIntro}.
Moreover, the same result holds for any function $f \in C^{n+1}(\mathbb{R})$ such that $f(t) = a(t)$ for $t \in \mathbb{R} \backslash [-1,1]$.
\end{thm}

Note that condition $\sum_{l=1}^n \frac{1}{p_l} = 1$ implies that $n=1$ if and only if for some $1 \leq k \leq n$ we have  $p_k = 1$. Further, if $n= 1$ then Theorem \ref{Thm=MainIntro} is the main result of \cite{CPSZ-JOT}. Therefore this paper mainly deals with the case $n \geq 2$ and $p_k >1$ for all $1 \leq k \leq n$.

Let us comment on the proof. In contrast to \cite{CPSZ-AJM}, which covers the case $n=1$, we do not rely on Calder\'on-Zygmund theory but rather rely on the key results from \cite{CPSZ-JOT} together with  a new reduction technique. Theorem \ref{Thm=Ultimate} shows that the problem of finding weak type $(1,1)$ estimates of a double operator integral of divided differences is concentrated on the case that $\boldA = (A, \ldots, A)$ with either $A \geq 0$ or $A \leq 0$.   To prove this we use reductions from multiple operator integrals to double operator integrals.

\vspace{0.3cm}

\noindent {\it Structure.} Section \ref{Sect=Prelim} settles all notation and preliminaries on divided differences and multiple operator integrals. Section \ref{Sect=ReductionFormula} introduces several reduction techniques for multiple operator integrals. Of crucial importance is Lemma \ref{lem=crucial}.  Section \ref{Sect=MainTheorem} proves a reduction theorem which is fundamental to our paper. Then in Section \ref{Sect=Consequences} we present the main results. In particular we prove Theorem \ref{Thm=MainIntro}.

\section{Preliminaries}\label{Sect=Prelim}
Throughout the entire paper $n$ is a fixed number in $\mathbb{N}_{\geq 1}$.
For sets $B_0 \subseteq B_1$ we write $B_1 \backslash B_0$ for the set of all elements in $B_1$ that are not in $B_0$.
We write $\chi_G$ for the indicator function of a set $G \subseteq \mathbb{R}^{n+1}$ and $\chi_0$ when $G = \{ 0 \}$.
For $p \in [1, \infty]$ we denote the conjugate exponent by  $p' \in [1, \infty]$ which is defined by $\frac{1}{p}+\frac{1}{p'} =1$.

For normed spaces $X$ and $Y$ we denote $X \times Y$ for the  Cartesian product equipped with the max norm $\Vert (x, y) \Vert = \max (\Vert x \Vert, \Vert y \Vert)$.

Inner products are linear in the first argument and anti-linear in the second one.
In this paper $H$ is an infinite dimensional separable Hilbert space, $B(H)$ is the algebra of all bounded operators on $H$ and $B(H)_{sa}$ stands for the set of all bounded self-adjoint operators. Note that the separability of $H$ is used in \cite{CLS}. We write ${\rm Tr}$ for the trace on $B(H)$.   For $A \in B(H)_{sa}$ we denote $\sigma(A)$ for the spectrum of $A$ and $\sigma_p(A)$ for the point spectrum of $A$. So $\sigma_p(A)$ consists of all eigenvalues of $A$. Let $E_A$ be the spectral measure of $A$ (see \cite{Rudin}). A scalar valued spectral measure of $A$ is   a positive scalar valued finite measure $\lambda_A$ on the Borel sets of $\sigma(A)$ having the same sets of measure 0 as $E_A$. As observed in \cite[Preliminaries]{CLS} such a measure always exists and the constructions below are independent of the choice of $\lambda_A$. See also \cite[Section IX.8]{Conway}.

\subsection{Schatten spaces $\cS_p$ and  operator ideals}  For addditional information concerning material reviewed in this subsection, we refer to \cite{LSZ}. We let $\cS_p(H), 1 \leq p  < \infty$ be the Schatten-von Neumann non-commutative $L_p$-spaces associated with $B(H)$. We simply write $\cS_p$ for $\cS_p(H)$ and omit $H$ in the notation.
$\cS_p$ is the Banach space consisting of all $x \in B(H)$ such that $\Vert x \Vert_p := {\rm Tr}(\vert x \vert^p)^{1/p} < \infty$.
 $\cS_\infty$ denotes the compact operators. The H\"older inequality holds $\Vert xy \Vert_p \leq \Vert x \Vert_q \Vert y \Vert_r$ whenever $x \in \cS_q, y \in \cS_r$ and $p^{-1} = q^{-1} + r^{-1}$.

For $x \in B(H)$ we set the singular value sequence
\[
\mu_k(x) = \inf \{  \Vert x (1-p) \Vert \mid  p \in B(H) \textrm{  projection},    {\rm Tr}(p)  \leq  k \},  \qquad k \in \mathbb{N}_{\geq 0}.
\]
We let $\cS_{1, \infty}$ be the space of $x \in B(H)$ for which $(\mu_k(x))_{k \in \mathbb{N}_{\geq 0}}$ is in $\ell_{1, \infty}$, e.g.
\[
\Vert x \Vert_{1, \infty} := \sup_{k \in \mathbb{N}_{\geq 0}} (k +1)\mu_k(x)  < \infty.
\]
Then $\cS_{1, \infty}$ is a quasi-Banach space with quasi-triangle inequality
\[
\Vert x + y \Vert_{1, \infty} \leq 2 \Vert x \Vert_{1, \infty} +  2 \Vert y \Vert_{1, \infty}, \qquad x,y \in \cS_{1, \infty}.
\]
Further $\cS_1 \subseteq \cS_{1, \infty} \subseteq \cS_p, 1 < p \leq \infty$.

\subsection{Multiple operator integrals}\label{Sect=MOI} Fix a separable Hilbert space $H$.  Let $\boldA = (A_0, \ldots, A_n)$ be an $(n+1)$-tuple of self-adjoint operators $A_i \in B(H)$. We shall write this as $\boldA \in B(H)_{sa}^{\times n+1}$.  Let $\lambda_{A_i}$ be the scalar valued spectral measure of $A_i$. For  functions $\phi_i \in L_\infty( \sigma(A_i), \lambda_{A_i}), i = 0, \ldots, n$ set $\phi = \phi_0 \otimes \ldots \otimes \phi_n$ and   define:
\[
T_\phi^{\boldA}: \overbrace{\cS_2 \times \ldots \times \cS_2}^{n \textrm{ copies}} \rightarrow \cS_2: (x_0, \ldots, x_n) \mapsto \phi_0(A_0) x_1 \phi_1(A_1) x_1 \ldots \phi_{n-1}(A_{n-1}) x_n  \phi_{n}(A_{n}).
\]
We have
\[
L_\infty( \sigma(A_0), \lambda_{A_0}) \otimes \ldots \otimes L_\infty( \sigma(A_n), \lambda_{A_n}) \subseteq L_\infty( \sigma(A_0) \times \ldots \times \sigma(A_n) , \lambda_{A_0} \times \ldots \times \lambda_{A_{n}}),
\]
by identifying $(\phi_0 \otimes \ldots \otimes \phi_n)(t_0, \ldots, t_n) = \phi_0(t_0)  \ldots  \phi_n(t_n), t_i \in \sigma(A_i)$. Next, the space $L_\infty( \sigma(A_0) \times \ldots \times \sigma(A_n) , \lambda_{A_0} \times \ldots \times \lambda_{A_{n}})$ is equipped with the weak-$\ast$-topology induced by the predual $L_1( \sigma(A_0) \times \ldots \times \sigma(A_n) , \lambda_{A_0} \times \ldots \times \lambda_{A_{n}})$ and the linear span of the  elementary tensor products   is  weak-$\ast$ dense in this space. In \cite{CLS} it is explained that also the space of bounded multi-linear maps  $\cS_2 \times \ldots \times \cS_2 \rightarrow \cS_2$ is canonically a dual space and therefore carries the weak-$\ast$ topology. More precisely,
for $\xi_0 \otimes \ldots \otimes \xi_{n} \otimes \eta \in \cS_2 \widehat{\otimes} \ldots \widehat{\otimes} \cS_2$ ($n+2$ projective tensor products) define the multi-linear map
\[
\overbrace{\cS_2 \times \ldots \times \cS_2}^{n +1 \textrm{ copies}} \rightarrow \cS_2:  \mu_0 \otimes \ldots  \otimes \mu_n \mapsto \langle \mu_0, \xi_0^\ast \rangle \cdots \langle \mu_n, \xi_n^\ast \rangle \eta.
\]
By \cite{CLS} this isomorphism complex linearly identifies the space of multi-linear maps $\cS_2 \times \ldots \times \cS_2 \rightarrow \cS_2$ as the dual of $\cS_2 \widehat{\otimes} \ldots \widehat{\otimes} \cS_2$ ($n+2$ tensors). This isomorphism defines the weak-$\ast$-topology on  $\cS_2 \times \ldots \times \cS_2 \rightarrow \cS_2$.

By \cite[Proposition 5]{CLS} (see also \cite[Section 4.2]{SkripkaTomskova} and \cite{Pavlov}, \cite{SolomyakStenkin}) the assignment $\phi \mapsto T_\phi^{\boldA}$ extends uniquely to a linear weak-$\ast$ continuous map:
\[
L_\infty( \sigma(A_0) \times \ldots \times \sigma(A_n) , \lambda_{A_0} \times \ldots \times \lambda_{A_{n}}) \rightarrow B(\cS_2 \times \ldots \times \cS_2, \cS_2).
\]
This defines $T_\phi^{\boldA}$ for $\phi$ in this domain.
By  \cite[Remarks after Corollary 10]{CLS} we have for such $\phi$ that
\[
\Vert \phi \Vert_\infty = \Vert T_\phi^{\boldA}:  \cS_2 \times \ldots \times \cS_2  \rightarrow \cS_2 \Vert.
\]
 Note that if $\phi: \mathbb{R}^{n+1} \rightarrow \mathbb{C}$ is a bounded Borel function then it defines an equivalence class $[\phi]$ in $L_\infty( \sigma(A_0) \times \ldots \times \sigma(A_n) , \lambda_{A_0} \times \ldots \times \lambda_{A_{n}})$ and we keep denoting $T_\phi^\boldA$ for $T_{[\phi]}^{\boldA}$.

Let $\phi: \mathbb{R}^{n+1} \rightarrow \mathbb{C}$ be a bounded Borel function and $\boldA \in B(H)_{sa}^{\times n+1}$. In this paper we shall be interested in extensions of $T_\phi^{\boldA}$ to various Schatten classes. Let $1 \leq p, p_1, \ldots, p_n < \infty$. We denote
\begin{equation}\label{Eqn=NormSp}
\Vert T_\phi^{\boldA}: \cS_{p_1} \times \ldots \times \cS_{p_n} \rightarrow \cS_p \Vert
\end{equation}
 for the infimum of all constants $C>0$ such that for every $(x_1, \ldots, x_n) \in (\cS_2 \cap \cS_{p_1}) \times \ldots \times (\cS_2 \cap \cS_{p_n})$   we have $T_\phi^{\boldA}(x_1, \ldots, x_n) \in \cS_p$ and moreover,
\[
\Vert T_\phi^{\boldA}(x_1, \ldots, x_n) \Vert_{\cS_p} \leq C \prod_{l=1, \ldots, n} \Vert x_l \Vert_{\cS_{p_l}}.
\]
In case \eqref{Eqn=NormSp} is finite    $T_\phi^\boldA$ extends to a bounded map $\cS_{p_1} \times \ldots \times \cS_{p_n}$ to $\cS_p$ still denoted by $T_\phi^\boldA$. Analogously we can replace the target space $\cS_p$ by $\cS_{1, \infty}$ in this terminology. We shall also say that
$ T_\phi^{\boldA}$ is bounded from $\cS_{p_1} \times \ldots \times \cS_{p_n}$ to $\cS_p$ or $\cS_{1, \infty}$.

\begin{lem}\label{Lem=Indicator}
Let  $\phi: \mathbb{R}^{n+1} \rightarrow \mathbb{C}$ be a bounded Borel function. Let $\chi_{+}$ (resp. $\chi_{-}$)  be the indicator function on $\mathbb{R}_{\geq 0}^{\times n+1}$ (resp. $-\mathbb{R}_{\geq 0}^{\times n+1}$).   Let $A \in B(H)_{sa}$.
\begin{enumerate}
\item\label{Item=One} If $A \geq 0$ we have
$T_{\phi}^{(A, \ldots, A)} =  T_{\phi \chi_+  }^{(A, \ldots, A)}$ and if $A \leq 0$ we have $T_{\phi}^{(A, \ldots, A)} =  T_{\phi \chi_-   }^{(A, \ldots, A)}$.
\item\label{Item=Two} $T_{\chi_{\pm}}^{(A, \ldots, A)}: \cS_{p_1} \times \ldots \times \cS_{p_n} \rightarrow \cS_1$ is a contraction for every $1 \leq p_1, \ldots, p_n < \infty$ with $\sum_{l=1}^n \frac{1}{p_l} = 1$.
\end{enumerate}
\end{lem}
\begin{proof}
Set the projection $P:= \chi_{[0, \infty)}(A)$.
\eqref{Item=One}
First assume that  $\phi = \phi_0 \otimes \ldots \otimes \phi_n \in L_\infty(\sigma(A_0), \lambda_{A_0} ) \otimes \ldots \otimes  L_\infty(\sigma(A_n), \lambda_{A_n} )$ is an elementary tensor product. If $A \geq 0,$ then   $P = 1$.  Then  for $x_i \in \cS_2$,
\[
\begin{split}
& T_{\phi}^{(A, \ldots, A)}(x_1, \ldots, x_n)
=
\phi_0(A) x_1 \phi_1(A)   \ldots \phi_{n-1}(A) x_n \phi_n(A) \\
= & \phi_0(A) P x_1 \phi_1(A) P    \ldots \phi_{n-1}(A)  P x_n \phi_n(A) P
=  T_{\phi \chi_{ + }}^{(A, \ldots, A)}(x_1, \ldots, x_n).
\end{split}
\]
By weak-$\ast$ density of the linear span of elementary products we conclude the lemma for every bounded Borel function $\phi: \mathbb{R}^{n+1} \rightarrow \mathbb{C}$. The statement for $A \leq 0$ follows similarly.
\eqref{Item=Two} We find for $x_l \in \cS_{p_l} \cap \cS_2$ that
\[
T_{\chi_+}^{(A, \ldots,  A)}(x_1, \ldots, x_n) =  P x_1  P   \ldots  P x_n  P,
\]
 which  defines a contraction by the H\"older inequality. The case for $\chi_-$ is treated similarly.
\end{proof}

\begin{rmk}
Let $\phi: \mathbb{R}^{n+1} \rightarrow \mathbb{C}$ be a bounded Borel function and let
\[
\phi_m  = \sum_{(l_0, \ldots, l_n) \in \mathbb{Z}^{n+1}}  \phi\left(   \frac{l_0}{m}  , \ldots,    \frac{l_n}{m}   \right) \chi_{ \prod_{i=0}^n [\frac{l_i}{m}, \frac{l_i+1}{m} ) }.
\]
Let $A \in B(H)_{sa}$  with spectral measure $E$ and set $E_{l,m} = E( [\frac{l}{m},  \frac{l+1}{m} ) )$.  Let $1 \leq p, p_1, \ldots, p_n < \infty$ be such that $\frac{1}{p} = \sum_{k=1}^n \frac{1}{p_k}$.  Then for $x_k \in \cS_2 \cap \cS_{p_k}$ we have
\begin{equation}\label{Eqn=TApprox}
T^{(A, \ldots, A)}_{\phi_m}(x_1, \ldots, x_n) =  \sum_{(l_0, \ldots, l_n) \in \mathbb{Z}^{n+1}}  \phi\left(   \frac{l_0}{m}  , \ldots,    \frac{l_n}{m}   \right) E_{l_0, m } x_1 E_{l_1, m} x_2 \ldots E_{l_{n-1}, m } x_n E_{l_n, m}.
\end{equation}

Assume that $\phi_m \rightarrow \phi$ pointwise (which holds true in particular if $\phi$ is continuous on $\mathbb{R}^{n+1} \backslash \{ 0 \}$). Then by the Lebesgue dominated convergence theorem $\phi_m \rightarrow \phi$ in the weak-$\ast$-topology of $L_\infty(\sigma(A), \lambda_A)^{\otimes n+1}$.
By weak-$\ast$-continuity we have a convergence in $\cS_2$, {\color{red} please, put an explanation here}
\[
T^{(A, \ldots, A)}_{\phi_m}(x_1, \ldots, x_n) \rightarrow T^{(A, \ldots, A)}_{\phi}(x_1, \ldots, x_n).
\]

On the other hand   assume that $\phi$ is in the class $\mathfrak{C}_n$ from \cite[Eqn. (3.1)]{PSS}. If the maps $T^{(A, \ldots, A)}_{\phi_m}$ are bounded $\cS_{p_1} \times \ldots \cS_{p_n} \rightarrow \cS_{p}$ uniformly in $m$ and if \eqref{Eqn=TApprox} converges in $\cS_p$ for every $x_i \in \cS_2 \cap \cS_{p_i}$, then we see that the limiting operator equals the one from \cite[Definition 3.1]{PSS}. In particular this applies to the class of polynomial integral momenta \cite[Theorem 5.3]{PSS} (where $1 < p < \infty$) and the multiple operator integrals appearing in Proposition \ref{Prop=EasyClass} below.

We conclude that  the multiple operator integrals that occur in this paper coincide with the ones defined in \cite[Definition 3.1]{PSS}. Further,  it was already observed in \cite[p. 510]{PSS} that these multiple operator integrals also agree with Peller's definition and approach \cite{PellerJFA} (see also \cite{Azamov}).
\end{rmk}

\subsection{Reduction to the case $\boldA = A^{\times n+1}$}
Let $H$ be a Hilbert space.
Consider $H^{(n+1)} = H \oplus \ldots \oplus H$ the $(n+1)$-fold direct sum and identify $B(H^{(n+1)}) \simeq M_{n+1}(\mathbb{C}) \otimes B(H)$ and $\cS_p(H^{(n+1)}) \simeq \cS_p(\mathbb{C}^{n+1}) \otimes \cS_p(H), 1 \leq p \leq \infty$. Let $E_{i,j} \in M_{n+1}(\mathbb{C})$ denote the matrix unit with zero entries everywhere except for a 1 at the $i$-th row and $j$-th column. We have the following matrix amplification trick.

\begin{prop}\label{Prop=Amplification}
Let $\phi: \mathbb{R}^{n+1} \rightarrow \mathbb{C}$ be a bounded Borel function. Let $\boldA = (A_0, \ldots, A_n) \in B(H)_{sa}^{\times n+1}$ and let $x_1, \ldots, x_n \in \cS_2$. Define elements of $M_{n+1}(\mathbb{C}) \otimes B(H)$ by
\[
\overrightarrow{A} = \sum_{l=0}^n E_{l,l} \otimes A_l, \qquad \textrm{ and } \qquad z_l = E_{l-1, l} \otimes x_l, \quad l = 1, \ldots, n.
\]
Note that $z_l \in \cS_2(\mathbb{C}^{n+1}) \otimes \cS_2(H)$ with $\Vert z_l \Vert_2 = \Vert x_l \Vert_2$.
Further set the $(n+1)$-tuple $\overrightarrow{\boldA} = (\overrightarrow{A}, \ldots, \overrightarrow{A})$.
Then,
\begin{equation}\label{Eqn=Amplification}
T^{\overrightarrow{\boldA}}_{\phi}(z_1, \ldots, z_n ) = E_{0,n} \otimes T^{\boldA}_{\phi}(x_1, \ldots, x_n).
\end{equation}
\end{prop}
\begin{proof}
By linearity and weak-$\ast$-continuity  of the maps
\[\phi\to T^{\boldA}_{\phi},\quad \phi\to T^{\overrightarrow{\boldA}}_{\phi},
\]
\noindent it suffices to check \eqref{Eqn=Amplification} for $\phi = \phi_0 \otimes \ldots \otimes \phi_n$ an elementary tensor product of bounded Borel functions $\phi_i: \mathbb{R} \rightarrow \mathbb{C}$.   We have,
\[
\phi_k(\overrightarrow{A}) = \sum_{l=0}^n E_{l,l} \otimes \phi_k(A_l).
\]
Thus,
\[
\phi_k(\overrightarrow{A}) z_{k+1} =
(E_{k,k} \otimes \phi_k(A_k)) (E_{k, k+1} \otimes x_{k+1}).
\]
Therefore,
\[
\begin{split}
T^{\overrightarrow{\boldA}}_{\phi}(z_1, \ldots, z_n ) = & \phi_0(\overrightarrow{A} ) z_1 \phi_1(\overrightarrow{A} ) \ldots  \phi_{n-1}(\overrightarrow{A} ) z_n \phi_{n}(\overrightarrow{A} )  \\
= &  (E_{0,0} \otimes \phi_0(A_0)) (E_{0,1} \otimes x_1) (E_{1,1} \otimes \phi_1(A_1)) (E_{1,2} \otimes x_2) \ldots  \\
 & \ldots (E_{n-1,n-1} \otimes \phi_{n-1}(A_{n-1} )) (E_{n-1,n} \otimes x_n) (E_{n,n} \otimes \phi_n(A_n)) \\
 = & E_{0,n}  \otimes (\phi_0(A_0) x_1 \phi_1(A_1) \ldots \phi_{n-1}(A_{n-1} ) x_n  \phi_n(A_n)) \\
 = & E_{0,n} \otimes T^{\boldA}_{\phi}(x_1, \ldots, x_n).
\end{split}
\]
This concludes the proof.
\end{proof}

\begin{cor} \label{Cor=Amplify}
Recall that $H$ is an infinite dimensional separable Hilbert space. Let $1 \leq p_1, \ldots, p_n <\infty$.  Let $\phi: \mathbb{R}^{n+1} \rightarrow \mathbb{C}$ be a Borel function. Suppose that there exists a constant $D>0$ such that for all $A \in B(H)_{sa}$ we have
\[
\Vert T_\phi^{(A, \ldots, A)}: \cS_{p_1} \times \ldots \times \cS_{p_n} \rightarrow \cS_{1, \infty} \Vert \leq D.
\]
Then in fact  for all $A_0, \ldots, A_n \in B(H)_{sa}$ we have
\[
\Vert T_\phi^{(A_0, \ldots, A_n)}: \cS_{p_1} \times \ldots \times \cS_{p_n} \rightarrow \cS_{1, \infty} \Vert \leq D.
\]
The same statement holds with the target space $\cS_{1, \infty}$ replaced by $\cS_r$ for $1 \leq r \leq \infty$.
\end{cor}
\begin{proof}
We use the notation of Proposition \ref{Prop=Amplification} and the fact that $H$ is isomorphic to $H^{(n+1)}$ because the dimension of $H$ is infinite. Note that for $l = 1, \ldots, n$ we have $\Vert x_l \Vert_{p_l} = \Vert z_l \Vert_{p_l}$. We thus have
\[
\begin{split}
& \Vert  T_\phi^{(A_0, \ldots, A_n)}(x_1, \ldots, x_n) \Vert_{\cS_{1, \infty}} =
\Vert E_{0,n} \otimes  T_\phi^{(A_0, \ldots, A_n)}(x_1, \ldots, x_n) \Vert_{\cS_{1, \infty}}\\
= &
\Vert T^{ (\overrightarrow{A}, \ldots, \overrightarrow{A})}_{\phi}(z_1, \ldots, z_n ) \Vert_{\cS_{1, \infty}}
\leq  D \prod_{l=1}^n \Vert z_l \Vert_{\cS_{p_l}}
  =   D \prod_{l=1}^n \Vert x_l \Vert_{\cS_{p_l}}.
\end{split}
\]
\end{proof}

\subsection{Reduction to the case $0 \not \in \sigma_p(A)$}
For $\phi: \mathbb{R}^{n+1} \rightarrow \mathbb{C}$ bounded Borel and $\delta \in \mathbb{R}$ let $\tau_\delta(\phi)(t_0, \ldots, t_n) = \phi(t_0+\delta, \ldots, t_n+\delta)$.
Recall that $\chi_{ 0  } := \chi_{ \{ 0 \} }$.
For $A \in B(H)_{sa}$ we have
\begin{equation}\label{Eqn=PhiTranslate}
T_{\tau_\delta(\phi)}^{(A, \ldots, A)} = T_{\phi}^{(A - \delta, \ldots, A - \delta)}, \qquad
T_{c \chi_{ 0  }}^{(A, \ldots, A)} = T_{c}^{( \chi_{0}( A)  , \ldots, \chi_0( A))},\quad c \in \mathbb{C}.
\end{equation}
Indeed,   one first verifies the first equality of \eqref{Eqn=PhiTranslate} on elementary tensors $\phi = \phi_0 \otimes \ldots \otimes \phi_n$ and then uses weak-$\ast$ density. The second equality of \eqref{Eqn=PhiTranslate} follows straight from the definitions.

\begin{prop}\label{Prop=ApproxZero}
Let $1 \leq p_1, \ldots, p_n < \infty$ be such that $\sum_{l=1}^n \frac{1}{p_l} = 1$.
Let $\phi: \mathbb{R}^{n+1} \rightarrow \mathbb{C}$ be a function that is continuous  on $\mathbb{R}^{n+1} \backslash \{ 0 \}$ and  with $\phi(0) = 0$ and $\Vert \phi \Vert_\infty \leq 1$. Suppose that there exists a constant $D >0$ such  that for all $A \in B(H)_{sa}$ with $0 \not \in \sigma_p(A)$ we have that
\begin{equation}\label{Eqn=ZeroApproxAssum}
\Vert T^{(A, \ldots, A)}_{\phi}: \cS_{p_1} \times \ldots \times \cS_{p_n} \rightarrow \cS_{1, \infty} \Vert \leq D.
\end{equation}
 Then for all $A \in B(H)_{sa}$  we have
 \[
\Vert T^{(A, \ldots, A)}_{\phi}: \cS_{p_1} \times \ldots \times \cS_{p_n} \rightarrow \cS_{1, \infty} \Vert \leq 2D+2.
\]
\end{prop}
\begin{proof}
For $\delta > 0$ let $\phi_\delta(t_0, \ldots, t_n) = \phi(t_0 + \delta, \ldots, t_n + \delta)$ if $(t_0, \ldots, t_n)$ is non-zero and $\phi_\delta(0, \ldots, 0) =0$. So
\begin{equation}\label{Eqn=PhiEx}
\phi_\delta = \tau_\delta( \phi) - \phi(\delta, \ldots, \delta)  \chi_{ 0 }.
\end{equation}
 We have that $\phi_\delta \rightarrow \phi$ uniformly on compact sets in $\mathbb{R}^{n+1} \backslash \{ 0 \}$ as $\delta \rightarrow 0$. Assume first that $A \in B(H)_{sa}$ is such that for some $\alpha >0$ we have spectral gap $\sigma( \vert A \vert) \subseteq \{ 0 \} \cup (\alpha, \infty)$.  Then $\phi_\delta \rightarrow \phi$ uniformly on $\sigma(A)^{ \times n+1}$. Therefore (see \cite[Remark after Corollary 10]{CLS}) for $x_l \in \cS_2 \cap \cS_{p_l}$ with $\Vert x_l \Vert_{\cS_{p_l}} \leq 1$  we have
\[
T_{\phi_\delta}^{ (A, \ldots, A) }(x_1, \ldots, x_n) \rightarrow T_{\phi}^{(A, \ldots, A) }(x_1, \ldots, x_n)
\]
 in norm of $\cS_2$. Further, for $\delta  >0$ we have
 \[
 T_{\phi_\delta}^{(A, \ldots, A)} \:\:
 =\!\!\!\!\!\!\!\!\!^{^{\eqref{Eqn=PhiEx}}}
  T_{\tau_\delta(\phi)}^{(A, \ldots, A)} -  T_{ \phi(\delta, \ldots, \delta)  \chi_{ 0  }}^{(A, \ldots,  A )}
  \:\:
 =
\!\!\!\!\!\!\!\!\!^{^{\eqref{Eqn=PhiTranslate}}}
 T_{\phi}^{(A-\delta, \ldots, A-\delta)} -  T_{ \phi(\delta, \ldots, \delta)  \chi_{ 0  }}^{(A, \ldots,  A )}.
 \]
 If  $0< \delta < \alpha$ then $0 \not \in \sigma_p(A - \delta)$,
 so that  by assumption \eqref{Eqn=PhiTranslate} and the quasi-triangle inequality,
\[
\begin{split}
\Vert T_{\phi_\delta}^{(A, \ldots, A)}(x_1, \ldots, x_n) \Vert_{\cS_{1, \infty}} \leq &  2 \Vert  T_{\phi}^{(A-\delta, \ldots, A-\delta)}(x_1, \ldots, x_n) \Vert_{\cS_{1, \infty}} +
2 \Vert  T_{ \phi(\delta, \ldots, \delta)   }^{(\chi_0(A), \ldots, \chi_0(A)  )}(x_1, \ldots, x_n) \Vert_{\cS_{1, \infty}} \\
\leq & 2 (D +   \phi(\delta, \ldots, \delta))
\leq  (2D + 2).
\end{split}
\]
By the Fatou property \cite{DDPS1} we have $T_{\phi}^{(A, \ldots, A) }(x_1, \ldots, x_n) \in \cS_{1, \infty}$ with norm majorized by $2D + 2$.

Now take general $A \in B(H)_{sa}$, not necessarily with spectral gap. Take again $x_l \in \cS_2 \cap \cS_{p_l}$ with $\Vert x_l \Vert_{\cS_{p_l}} \leq 1$. For $\alpha > 0$ set
\[
P_\alpha := \chi_{(-\infty, - \alpha) \cup \{ 0 \} \cup (\alpha, \infty)}(A).
\]
We have as $\alpha \searrow 0$ that $P_{\alpha} x_l P_{\alpha} \rightarrow x_l$ both in the  norm of $\cS_{2}$ and $\cS_{p_l}$, see \cite{ChilSuk-JOT}. It follows that for $\alpha \searrow 0$,
\[
T_{\phi}^{(A, \ldots, A)} ( P_\alpha x_1 P_\alpha, \ldots, P_\alpha x_n P_\alpha) \rightarrow
T_{\phi}^{(A, \ldots, A)} (  x_1 , \ldots,  x_n ),
\]
in the norm of $\cS_2$.
Further,
\[
  T_{\phi}^{(A, \ldots, A)} ( P_\alpha x_1 P_\alpha, \ldots, P_\alpha x_n P_\alpha)   =   T_{\phi}^{(P_\alpha A, \ldots, P_\alpha A)} ( P_\alpha x_1 P_\alpha, \ldots, P_\alpha x_n P_\alpha),
\]
and the right hand side of this expression is in $\cS_{1, \infty}$ with norm majorized by $2D+2$. By the Fatou property \cite{DDPS1} we conclude that  $ T_{\phi}^{(A, \ldots, A)} ( x_1 , \ldots, x_n ) \in \cS_{1, \infty}$ with norm majorized by $2D+2$.

\end{proof}

\subsection{Divided differences}
 Let $C^n(I)$ be the set of all  $n$ times continuously differentiable functions on  $I$.
For $g \in C^{n}(\mathbb{R})$ let $g^{(n)}$ be the $n$-th order derivative of $g$.
Let $\Cn$ be the space of functions $f \in C^{n-1}(\mathbb{R})$ whose restriction to $\mathbb{R} \backslash \{ 0 \}$ is in $C^n(\mathbb{R} \backslash \{ 0 \})$.  For $f \in \Cn$ we set the $n$-th order divided difference function $f^{[k,n]}: \mathbb{R}^{k+1} \rightarrow \mathbb{C}$ by inductively defining  for $0 \leq k \leq n$ the function, (here, $f^{[0,n]}=f$)
\begin{equation}\label{Eqn=DividedDifference}
f^{[ k, n ]}(t_0, \ldots, t_k) =
\left\{
\begin{array}{ll}
\frac{ f^{[ k-1, n ]}(t_0, t_2, t_3, \ldots, t_k) -  f^{[ k-1, n ]}(t_1,  t_2, t_3  \ldots, t_k) }{ t_0  - t_1} &t_0 \not = t_1, \\
0 &  k=n  \textrm{ and }  0 = t_0 = t_1 = \ldots = t_n,  \\
\left. \frac{d}{d t}\right|_{t= t_0}  f^{[ k-1, n ]}(t, t_2, \ldots, t_k) & \textrm{otherwise}.
\end{array}
\right.
\end{equation}
Since $f$ is $n-1$ times differentiable on $\mathbb{R}$  and $n$ times differentiable on $\mathbb{R} \backslash \{ 0 \}$ the formulae  \eqref{Eqn=DividedDifference} are well-defined. Further,  $f^{[k,n]}, 0 \leq k < n$ is continuous on $\mathbb{R}^k$ and $f^{[n,n]}$ is continuous on $ \mathbb{R}^{n+1} \backslash \{ 0 \}$.
 For $k=n$ this definition of the divided difference function differs from the usual one (as in \cite{PSS}) in the point $0 \in \mathbb{R}^{n+1}$; the  conventional definition in our current notation would be $f^{[n,n]} + f^{(n)}(0) \chi_0$   (which requires $f$ to be in $C^{n}(\mathbb{R})$). We have  that $f^{[k,n]}$ is symmetric under permutation of the variables (see \cite{DevoreLorentz}); i.e. for any permutation $\sigma$ of $\{0, \ldots, k\}$ we have
\begin{equation}\label{Eqn=Sym}
f^{[k,n]}(t_0, \ldots, t_k) = f^{[k,n]}(t_{\sigma(0)}, \ldots, t_{\sigma(k)}).
\end{equation}
In this paper we shall fix $n$ and write
\[
f^{[ k]} := f^{[ k,n]}, \qquad 0 \leq k \leq n.
\]
The following result follows from proofs and observations that were made in \cite{PSS}.

\begin{prop}\label{Prop=EasyClass}
Let $g \in C^{n+1}(\mathbb{R})$ have compact support. Then there exists a constant $D >0$ such that for every $\boldA \in B(H)_{sa}^{\times n+1}$ we have
\[
\Vert T_{g^{[n,n]}}^{\boldA}: \cS_{p_1} \times \ldots \times \cS_{p_n} \rightarrow \cS_1 \Vert \leq D.
\]
\end{prop}
\begin{proof}
The conditions on $g$ imply that the Fourier transform of the $n$-th order derivative $g^{(n)}$ is integrable \cite[Lemma 7]{PSCrelle} or \cite[Top of p. 503]{PSS}. By \cite[Lemma 2.3]{Azamov} (see \cite[Top of p. 512]{PSS}) $g^{[n,n]} + g^{(n)} \chi_0$ and hence $g^{[n,n]}$ belongs to the class $\mathfrak{C}_n$ defined in \cite{PSS}. So by \cite[Lemma 3.5]{PSS} we conclude the argument.
\end{proof}

\section{A reduction formula for divided differences}\label{Sect=ReductionFormula}

The aim of this section is to demonstrate   reduction techniques for multiple operator integrals. In particular Lemma \ref{lem=crucial} is crucial in this paper.

\subsection{A special double operator integral}
Define the auxiliary functions
\[
\rho(s) = \frac{\vert s_0 \vert}{ \vert s_0 \vert + \vert s_1 \vert}, \qquad \psi(s) = \frac{\vert s_1 \vert}{\vert s_0 \vert + \vert s_1 \vert}, \qquad s = (s_0, s_1) \in \mathbb{R}^2\backslash \{ (0,0) \}.
\]
Further $\rho(0,0) = \psi(0,0) =  0$.
The following lemma is the main tool behind the paper \cite{CPSZ-JOT} and it is implicitly stated and proved there. We show how to derive it from \cite{CPSZ-JOT} in the discrete case and then refer to \cite{CPSZ-AJM} for an approximation argument. The lemma can also be derived from the much stronger result \cite[Theorem 1.2]{CPSZ-AJM}.

\begin{lem}[\cite{CPSZ-JOT}]\label{Lem=PsiBoundNew}
There exists $C >0$ such that for every $A_0, A_1 \in B(H)_{sa}$ we have,
\begin{equation}\label{Eqn=WeakAbs}
\Vert T_{\rho}^{A_0, A_1}: \cS_{1} \rightarrow \cS_{1, \infty} \Vert    \leq C.
\end{equation}
 The same statement is true with $\rho$ replaced by $\psi$.
\end{lem}
\begin{proof}[Proof sketch]
By Corollary \ref{Cor=Amplify} we may assume that $A_0 = A_1 = A \in B(H)_{sa}$. By Proposition \ref{Prop=ApproxZero} we may assume that $0 \not \in \sigma_p(A)$.
For $\epsilon_1, \epsilon_2 \in \{ -, +\}$ let $\chi_{\epsilon_1 \epsilon_2}$ be the indicator function of $\epsilon_1 \mathbb{R}_{\geq 0} \times \epsilon_2 \mathbb{R}_{\geq 0}$. Under these assumptions it follows from the definition of the double/multiple operator integral that
\[
T^{A, A}_{\rho} = T_{ \chi_{++} \rho}^{A, A} +  T_{ \chi_{+-} \rho}^{A, A} +  T_{ \chi_{-+} \rho}^{A, A} + T_{ \chi_{--} \rho}^{A, A}.
\]
Hence it suffices to estimate the norm of each of the latter four summands.   Note that $\rho(s_0, s_1) = \rho( \vert s_0 \vert ,  \vert s_1 \vert )$. We have
\[
T_{  \chi_{--}  \rho}^{-A, -A}  = T_{\chi_{-+} \rho}^{-A, A}  = T_{\chi_{+-} \rho}^{A, -A} = T_{\chi_{++} \rho}^{A, A},
\]
and so it suffices to estimate $T_{\chi_{++} \rho}^{A, A}$.   Setting  $A_+ = \chi_{[0, \infty)}(A)$
we have   that  $T_{\chi_{++} \rho}^{A, A} = T_{\chi_{++} \rho}^{A_+, A_+}$ so that we may assume without loss of generality that $A$ has non-negative spectrum.

Assume further that $A$ has finite spectrum and $0 \not \in \sigma(A)$. So $A = \sum_{k=1}^K \lambda_k q_k$ with $\lambda_k > 0$ and $q_k$ the spectral projections. Then,
\[
T_{\rho}^{A, A}(x) = \sum_{k,l = 1}^K  \frac{\lambda_k}{\lambda_k + \lambda_l} q_k x q_l = \frac{1}{2} \sum_{k,l = 1}^K  \left( 1 +   \frac{\lambda_k - \lambda_l}{\lambda_k + \lambda_l} \right) q_k x q_l .
\]
Then \cite[Lemma 3.2]{CPSZ-JOT} shows that $T_{\rho}^{A, A}$ is bounded $\cS_1\rightarrow \cS_{1, \infty}$ with bound uniform in  $A\in B(H)_{sa}$.  For $A \geq 0$ arbitrary we have that  $T_{\rho}^{A, A}: \cS_1\rightarrow \cS_{1, \infty}$ uniformly boundedly in $A$ by approximation (see \cite[Section 5]{CPSZ-AJM}).
Since $\psi = 1 -\phi$ the last statement of the lemma follows from the others.
\end{proof}

\subsection{A reduction formula for divided differences}

\begin{lem}\label{Lem=PreCrucial}
Let $f \in \Cn$. We have
\begin{equation}\label{Eqn=DivDifRed}
f^{[n]}(0, t_1, \ldots, t_n) = g^{[n-1]}(t_1, t_2, \ldots, t_n),
\end{equation}
where $g(t) = f^{[1]}(t,0)$ and $t , t_1, \ldots, t_n \in \mathbb{R}$. Here $f^{[n]} = f^{[n,n]}$ and $g^{[n-1]}= g^{[n-1, n-1]}$.
\end{lem}
\begin{proof}
If all $t_i'$s are 0,  then  \eqref{Eqn=DivDifRed} follows from \eqref{Eqn=DividedDifference}. So assume not all $t_i'$s are 0. For $n = 1$ we obtain
\[
f^{[1]}(0,t_1) = \frac{f(0) - f(t_1) }{ 0 - t_1} = g(t_1).
\]
We proceed by induction and suppose that the assertion holds for $n$. We prove it for $n+1$. Assume $t_n \not = t_{n+1}$.
We have by definition \eqref{Eqn=DividedDifference} and permutation invariance of the variables \eqref{Eqn=Sym} that
$$
f^{[n+1]}(0, t_1, \ldots, t_{n+1})\stackrel{\eqref{Eqn=Sym}}{=}f^{[n+1]}(t_n,t_{n+1},0, t_1, \ldots, t_{n-1})=$$
$$\stackrel{\eqref{Eqn=DividedDifference}}{=}\frac{f^{[n]}(t_n,0, t_1, \ldots, t_{n-1}) - f^{[n]}(t_{n+1},0, t_1, \ldots, t_{n-1})}{t_n - t_{n+1}}=$$
$$\stackrel{\eqref{Eqn=Sym}}{=}\frac{f^{[n]}(0, t_1, \ldots, t_{n-1}, t_n) - f^{[n]}(0, t_1, \ldots, t_{n-1}, t_{n+1} )  }{t_n - t_{n+1}}.$$
By induction
\[
\begin{split}
f^{[n]}(0, t_1, \ldots, t_{n-1}, t_n) = & g^{[n-1]}(t_1, \ldots, t_{n-1}, t_n), \\
f^{[n]}(0, t_1, \ldots, t_{n-1}, t_{n+1}) = & g^{[n-1]}(t_1, \ldots, t_{n-1}, t_{n+1}). \\
\end{split}
\]
Hence
\[
f^{[n+1]}(0, t_1, \ldots, t_{n+1}) =
\frac{g^{[n-1]}(t_1, \ldots, t_{n-1}, t_n) - g^{[n-1]}(t_1, \ldots, t_{n-1}, t_{n+1}) }{t_n - t_{n+1}}
= g^{[n]}(t_1, \ldots, t_{n+1}).
\]
Finally, if  $t_n  = t_{n+1}$ and not all $t_i$'s are 0, then \eqref{Eqn=DivDifRed} follows by continuity from the previous cases.
\end{proof}

The following formula shall be crucial in the proof of our main theorem.  It contains a new decomposition of $f^{[n]}$ as a linear combination of a product of a function of $2$ variables and a function of $n$ variables.

\begin{lem}\label{lem=crucial}
 Let $f \in \Cn$. We have for every $0 \leq i , j \leq n, i \not = j$  and every $t \in \mathbb{R}^{n+1} \backslash \{ 0 \}$ with $t_i \not = t_j$ and $t_i \not = 0, t_j \not = 0$ that,
\begin{equation}\label{Eqn=Decompose}
\begin{split}
f^{[n]}(t_0, \ldots, t_n) = &
\frac{t_i}{t_i - t_j} g^{[n-1]}(t_0, \ldots, t_{j-1},    t_{j+1}, \ldots, t_n ) - \frac{t_j}{ t_i - t_j}  g^{[n-1]}(t_0, \ldots, t_{i-1},   t_{i+1}, \ldots, t_n ).
\end{split}
\end{equation}
Here $g(t) = f^{[1]}(t,0), t \in \mathbb{R}$. Here $f^{[n]} = f^{[n,n]}$ and $g^{[n-1]}= g^{[n-1, n-1]}$.
 \end{lem}
\begin{proof}
Since $t_i \not = t_j$ not all variables are equal, so we are not in the second case of the defining relation for $f^{[n]}$, see \eqref{Eqn=DividedDifference}.
By \eqref{Eqn=Sym} we have
\[
f^{[n-1]}(t_0, \ldots, t_{i-1},  t_{i+1}, \ldots, t_{j-1}, 0 ,t_{j+1} \ldots,    t_{n}) = f^{[n-1]}(t_0, \ldots,  t_{i-1}, 0, t_{i+1} \ldots, t_{j-1},   t_{j+1}, \ldots,    t_{n}).
\]
We have by using \eqref{Eqn=DividedDifference} and \eqref{Eqn=Sym} for the first and last equality,
\[
\begin{split}
& f^{[n]}(t_0, \ldots, t_n)\\
 = & \frac{ f^{[n-1]}(t_0, \ldots,t_{j-1}, t_{j+1} \ldots  t_{n})  - f^{[n-1]}(t_0, \ldots,t_{i-1}, t_{i+1} \ldots  t_{n})  }{t_i - t_j} \\
= &
  \frac{t_i}{t_i - t_j}  \frac{  f^{[n-1]}(t_0, \ldots,t_{j-1}, t_{j+1} \ldots  t_{n}) - f^{[n-1]}(t_0, \ldots, t_{i-1},  t_{i+1}, \ldots, t_{j-1}, 0 ,t_{j+1} \ldots,    t_{n})}{t_i} \\
 & + \frac{t_j}{ t_i - t_j}   \frac{ f^{[n-1]}(t_0, \ldots,  t_{i-1}, 0, t_{i+1} \ldots, t_{j-1},   t_{j+1}, \ldots,    t_{n})  - f^{[n-1]}(t_0, \ldots,t_{i-1}, t_{i+1} \ldots  t_{n})    }{t_j} \\
  = &  \frac{t_i}{t_i - t_j} f^{[n]}(t_0, \ldots, t_{j-1}, 0,  t_{j+1}, \ldots, t_n ) - \frac{t_j}{ t_i - t_j}  f^{[n]}(t_0, \ldots, t_{i-1}, 0,  t_{i+1}, \ldots, t_n ).
 \end{split}
\]

\end{proof}

Let $f \in \Cn$.  Now define inductively $f_0 = f$ and then
\begin{equation}\label{Eqn=Fl}
f_l(t) = f_{l-1}^{[1]}(t,0), \qquad t \in \mathbb{R}, 1 \leq l < n.
\end{equation}

\begin{lem}\label{Lem=PostCrucial}
Let $f \in \Cn$. We have for $l=0,1, \ldots, n-1$,
\begin{equation}
f^{[n]}( t_0, \ldots, t_{n-l}, 0, \ldots, 0) = f^{[n-l]}_{l}(t_{0},  \ldots, t_{n-l}),
\end{equation}
where $t_0,  \ldots, t_{n-l} \in \mathbb{R}$. Here $f^{[n]} = f^{[n,n]}$ and $f_l^{[n-l]}= f_l^{[n-l, n-l]}$.
\end{lem}
\begin{proof}
The proof follows by induction on $l$. If $l = 0$ the statement is trivial.  Suppose that the corollary is proved for $l$. We shall prove it for $l+1$. Indeed by induction, symmetry of the variables \eqref{Eqn=Sym} and  Lemma \ref{Lem=PreCrucial} applied to the function $f^{[n-l]}_{l}$ we have
\[
\begin{split}
& f^{[n]}( t_0, \ldots, t_{n-l-1}, 0, \ldots, 0) =   f^{[n-l]}_{l}( t_{0},  \ldots, t_{n-l-1}, 0) \\
= &
f^{[n-l]}_{l}(0, t_{0},  \ldots, t_{n-l-1}) =
 f^{[n-l-1]}_{l+1}(  t_{0},  \ldots, t_{n-l-1}  ).
\end{split}
\]
\end{proof}

\section{Main results}\label{Sect=MainTheorem}

The aim of this section is the proof of  Theorem \ref{Thm=Ultimate}. The theorem remarkably reduces the problem of estimating multiple operator integrals of divided differences $T_{f^{[n]}}^{(A_0, \ldots, A_n)}$ to the case that $A := A_0 = \ldots = A_n$ and $A \geq 0$.  We shall see in Section \ref{Sect=Consequences} that for functions that are close to the generalized absolute value map this reduction is sufficient to obtain weak type $(1,1)$ estimates.

\subsection{Main theorem in a special case}   We assume in this subsection that $A := A_0 = \ldots = A_n$ is in $B(H)_{sa}$ with $0 \not \in \sigma_p(A)$  where $\sigma_p(A)$ is the point spectrum of $A$.
Let $1<p_1, \ldots, p_n< \infty$ such that $\sum_{k=1}^n \frac{1}{p_k} = 1$. Take $x_k \in \cS_{p_k} \cap \cS_2$. Let $\mathcal{A} \subseteq \{ 0, \ldots, n\}$ and set
\[
x_k^{\mathcal{A}} =
\left\{
\begin{array}{ll}
\chi_{(0, \infty)}(A) x_k  \chi_{(0, \infty)}(A),& k-1,k \in \mathcal{A}, \\
\chi_{(0, \infty)}(A) x_k  \chi_{(-\infty, 0)}(A), & k-1 \in \mathcal{A}, k \not \in \mathcal{A}, \\
\chi_{(-\infty, 0)}(A) x_k \chi_{(0, \infty)}(A), & k-1 \not \in \mathcal{A}, k \in \mathcal{A}, \\
\chi_{(-\infty, 0)}(A) x_k \chi_{(-\infty, 0)}(A), & k-1, k \not \in \mathcal{A}. \\
\end{array}
\right.
\]
Since we assumed $0 \not \in \sigma_p(A)$ we have $1 = P_- + P_+$ with $P_- = \chi_{-(\infty, 0)}(A)$ and $P_+ =  \chi_{(0, \infty)}(A)$.
 Let $f \in \Cn$ and assume that $f^{[n]}$ is bounded.
Since $0 \not \in \sigma_p(A)$ multi-linearity of the multiple operator integral gives
\begin{equation}\label{Eqn=TADecomposition}
\begin{split}
& T^{(A, \ldots, A)}_{f^{[n]}}( x_1, \ldots, x_n) \\
= &
T^{(A, \ldots, A)}_{f^{[n]}}( ( P_- + P_+ ) x_1 ( P_- + P_+ ), \ldots, ( P_- + P_+ ) x_n ( P_- + P_+ ) ) \\
= &
 \sum_{\mathcal{A} \subseteq \{0, \ldots, n\}}  T^{(A, \ldots, A)}_{f^{[n]}}(x_1^{\mathcal{A}}, \ldots, x_n^{\mathcal{A}}).
\end{split}
\end{equation}
The $\cS_{1, \infty}$-norms of these summands where $\mathcal{A} \not = \emptyset, \mathcal{A} \not = \{0, \ldots, n\}$ turn out to be much easier to estimate.
  Recall the auxiliary functions
\[
\rho(s) = \frac{\vert s_0 \vert}{ \vert s_0 \vert + \vert s_1 \vert}, \qquad \psi(s) = \frac{\vert s_1 \vert}{\vert s_0 \vert + \vert s_1 \vert}, \qquad s = (s_0, s_1) \in \mathbb{R}^2\backslash \{ (0,0) \}.
\]
Further $\rho(0,0) = \psi(0,0) =  0$.
Further, recall that $f_l$ was defined in \eqref{Eqn=Fl}.

\begin{lem}\label{Lem=Induction}
Let $f \in \Cn$ with $f^{[n]}$ bounded, let $A \in B(H)_{sa}, 0 \not \in \sigma_p(A)$ and let  $1<p_1, \ldots, p_n< \infty$ be such that $\sum_{k=1}^n \frac{1}{p_k} = 1$.
Suppose that  $\mathcal{A} \not = \emptyset, \mathcal{A}  \not = \{0, \ldots, n\}$. Then there exists an absolute constant $C >0$ such that for $x_k \in \cS_{p_k} \cap \cS_2$,
\begin{equation}\label{Eqn=FGEstimate}
\begin{split}
\Vert T^{(A, \ldots, A)}_{f^{[n]}}(x_1^{\mathcal{A}}, \ldots, x_n^{\mathcal{A}})  \Vert_{\cS_{1, \infty}}
\leq & C \max_{1 \leq k \leq n} \{ p_k, p_k' \}  \cdot M(f;A) \cdot  \prod_{k=1}^n \Vert x_k \Vert_{\cS_{p_k}}.
\end{split}
\end{equation}
Where $M(f; A)$ is the maximum of the terms:
\begin{equation}\label{Eqn=MExpression}
\begin{split}
 &\max_{1 \leq k \leq n-1} \Vert T^{(A, \ldots, A)}_{f^{[n-1]}_1} : \cS_{p_1} \times \ldots \times \cS_{p_{k-1}} \times \cS_{ \frac{p_k p_{k+1}}{p_k + p_{k+1}} } \times \cS_{p_{k+2}} \times \ldots \times \cS_{p_n} \rightarrow \cS_{1, \infty}\Vert, \\
& \Vert T^{(A, \ldots, A)}_{f^{[n-1]}_1} : \cS_{p_1} \times \ldots  \times \cS_{p_{n-1}} \rightarrow \cS_{q_1} \Vert,
   \Vert T^{(A, \ldots, A)}_{f^{[n-1]}_1} : \cS_{p_2} \times \ldots  \times \cS_{p_n} \rightarrow \cS_{q_2} \Vert,   \\
\end{split}
\end{equation}
with  $\frac{1}{q_1} = \sum_{i=1}^{n-1} \frac{1}{p_i}, \frac{1}{q_2} = \sum_{i=2}^{n} \frac{1}{p_i}$.
\end{lem}
\begin{proof}
Since we may multiply $f$ with a positive scalar we may assume without loss of generality that $M(f;A) = 1$.

Throughout the entire proof, fix $0 \leq k < n$ such that $k \in \mathcal{A}$ and $k+1 \not \in \mathcal{A}$ (or $k+1\in \mathcal{A}$ and $k\not \in \mathcal{A}$). We assume the first case, the second case can be proved similarly.
We then have 
$$x_{k+1}^{\mathcal{A}}=\chi_{(0,\infty)}(A)x_{k+1}^{\mathcal{A}}\chi_{(-\infty, 0 )}(A)$$
and, therefore,  
$$T_{f^{[n]}}^{(A, \ldots, A)}(x_1^{\mathcal{A}}, \ldots, x_n^{\mathcal{A}} ) =T_{f^{[n]}}^{(A, \ldots, A)}(x_1^{\mathcal{A}}, \ldots, x_k^{\mathcal{A}},\chi_{(0,\infty)}(A)x_{k+1}^{\mathcal{A}}\chi_{(-\infty, 0 )}(A),x_{k+2}^{\mathcal{A}},\cdots,x_n^{\mathcal{A}} )=$$
$$=T_{f^{[n]}}^{(\overbrace{A, \ldots, A}^{k \textrm{ terms}}, \chi_{(0,\infty)}(A) A,  \chi_{(-\infty, 0 )}(A)A, \overbrace{A, \ldots, A}^{n-k-1 \textrm{ terms}})}(x_1^{\mathcal{A}}, \ldots, x_n^{\mathcal{A}} ).
$$
Consequently, this expression only depends on the values $f^{[n]}(t_0, \ldots, t_n)$ with $t_k >0$ and $t_{k+1} <0$.
For $t_k > 0$ and $t_{k+1} < 0$ we have by Lemma \ref{lem=crucial} that
\[
\begin{split}
f^{[n]}(t_0, \ldots, t_n)
= &
\rho(t_k, t_{k+1}) f^{[n-1]}_1(t_0, \ldots, t_{k-1}, t_k, t_{k+2}, \ldots, t_n)\\
& +
\psi(t_k, t_{k+1}) f^{[n-1]}_1(t_0, \ldots, t_{k-1}, t_{k+1}, t_{k+2}, \ldots, t_n).
\end{split}
\]
Therefore if $0 < k  < n-1$ we have\footnote{ Here, we are using the equality
$$T^{(A,\cdots,A)}_{h_1}(V_1,\cdots,V_n)=T^{(A,\cdots,A)}_{h_2}(V_1,\cdots,V_k,T^{(A,A)}_{h_3}(V_{k+1}),V_{k+2},\cdots,V_n),$$
which is valid whenever
$$h_1(t_0,\cdots,t_n)=h_3(t_k,t_{k+1})\cdot h_2(t_0,\cdots,t_n),$$
and the observation
$$T^{\overbrace{(A,\cdots,A)}^{n+1 \textrm{ terms}}}_h(V_1,\cdots,V_n)=T^{\overbrace{(A,\cdots,A)}^{n \textrm{ terms}}}_h(V_1,\cdots,V_k,V_{k+1}\cdot V_{k+2},V_{k+3},\cdots,V_n),$$
which is valid whenever $h$ does not depend on the $(k+1)$-st variable. These equalities can be verified directly when $h = h_0 \otimes \ldots \otimes h_n$ is an elementary tensor product; then the general case follows from weak-$\ast$-continuity just as in the proofs of Lemma \ref{Lem=Indicator} and Proposition \ref{Prop=Amplification}.   See also \cite[Lemma 3.2]{PSS}.
} that
\[
\begin{split}
T_{f^{[n]}}^{(A, \ldots, A)}(x_1^{\mathcal{A}}, \ldots, x_n^{\mathcal{A}} )
= &
T_{f_1^{[n-1]}}^{(A, \ldots, A)}(x_1^{\mathcal{A}}, \ldots, x_{k-1}^{\mathcal{A}}, x_k^{\mathcal{A}}, T_\rho^{A,A}(x_{k+1}^{\mathcal{A}}) \cdot x_{k+2}^{\mathcal{A}},  x_{k+3}^{\mathcal{A}},\ldots,    x_n^{\mathcal{A}} ) \\
& + T_{f_1^{[n-1]}}^{(A, \ldots, A)}(x_1^{\mathcal{A}}, \ldots, x_{k-1}^{\mathcal{A}}, x_k^{\mathcal{A}} \cdot  T_\psi^{A,A}(x_{k+1}^{\mathcal{A}}), x_{k+2}^{\mathcal{A}}, \ldots,    x_n^{\mathcal{A}} ).
\end{split}
\]
In case $k = 0$ we have
\begin{equation}\label{Eqn=One}
T_{f^{[n]}}^{(A, \ldots, A)}(x_1^{\mathcal{A}}, \ldots, x_n^{\mathcal{A}} ) = T_{f_1^{[n-1]}}^{(A, \ldots, A)}( T_\rho^{A}(x_{1}^{\mathcal{A}}) \cdot  x_2^{\mathcal{A}}, \ldots, x_{n}^{\mathcal{A}} )
+  T_\psi^{A,A}(x_{1}^{\mathcal{A}})  \cdot T_{f_1^{[n-1]}}^{(A, \ldots, A)}(x_2^{\mathcal{A}}, \ldots,  x_{n}^{\mathcal{A}}   ).
\end{equation}
In case $k = n-1$ we have
\begin{equation}\label{Eqn=Two}
T_{f^{[n]}}^{(A, \ldots, A)}(x_1^{\mathcal{A}}, \ldots, x_n^{\mathcal{A}} ) =
  T_{f_1^{[n-1]}}^{(A, \ldots, A)}(x_1^{\mathcal{A}}, \ldots, x_{n-1}^{\mathcal{A}} ) \cdot T_\rho^{A,A}(x_{n}^{\mathcal{A}})
+  T_{f_1^{[n-1]}}^{(A, \ldots, A)}(x_1^{\mathcal{A}}, \ldots, x_{n-2}^{\mathcal{A}}, x_{n-1}^{\mathcal{A}} \cdot  T_\psi^{A,A}(x_{n}^{\mathcal{A}}) ).
\end{equation}
Let us first consider the case $0 < k < n-1$; the cases $k = 0,  n-1$ can be proved similarly. We find from the quasi-triangle inequality and the assumption $M(f; A) \leq 1$ that
 \[
 \begin{split}
  \Vert T_{f^{[n]}}^{(A, \ldots, A)}(x_1^{\mathcal{A}}, \ldots, x_n^{\mathcal{A}}) \Vert_{\cS_{1, \infty}}
  \leq & 2 \Vert  T^{(A, \ldots, A)}_{f_1^{[n-1]}}(x_1^{\mathcal{A}}, \ldots, x_{k-1}^{\mathcal{A}}, x_k^{\mathcal{A}}, T_\rho^{A,A}(x_{k+1}^{\mathcal{A}}) x_{k+2}^{\mathcal{A}}, x_{k+3}^{\mathcal{A}}, \ldots, x_n^{\mathcal{A}} )  \Vert_{\cS_{1, \infty}} \\
  & + 2 \Vert  T_{f_1^{[n-1]}}^{(A, \ldots, A)}(x_1^{\mathcal{A}}, \ldots, x_{k-1}^{\mathcal{A}}, x_k^{\mathcal{A}}  T_\psi^{A,A}(x_{k+1}^{\mathcal{A}}  ), x_{k+2}^{\mathcal{A}}, \ldots, x_n^{\mathcal{A}}    ) \Vert_{\cS_{1, \infty}}\\
  \leq & 2 \prod_{l=1}^k \Vert x_l^{\mathcal{A}} \Vert_{\cS_{p_l}}   \Vert T_\rho^{A,A}(x_{k+1}^{\mathcal{A}})  x_{k+2}^{\mathcal{A}}  \Vert_{\cS_{ \frac{p_{k+1}p_{k+2} }{p_{k+1} + p_{k+2}} } } \prod_{l=k+3}^n \Vert x_l^{\mathcal{A}} \Vert_{\cS_{p_l}} \\
  & +  2 \prod_{l=1}^{k-1} \Vert x_l^{\mathcal{A}} \Vert_{\cS_{p_l}}   \Vert  x_{k}^{\mathcal{A}}  T_\psi^{A,A}(x_{k+1}^{\mathcal{A}})  \Vert_{\cS_{ \frac{p_{k}p_{k+1} }{p_{k} + p_{k+1}} } } \prod_{l=k+2}^n \Vert x_l^{\mathcal{A}} \Vert_{\cS_{p_l}} \\
  \leq & 2 \prod_{l=1}^n \Vert x_l \Vert_{\cS_{p_l}}  \left( \Vert T_\rho^{A,A} :  \cS_{p_{k+1}} \rightarrow  \cS_{p_{k+1}} \Vert  + \Vert T_\psi^{A,A}:  \cS_{p_{k+1}} \rightarrow  \cS_{p_{k+1}} \Vert \right).
 \end{split}
 \]
 By Lemma \ref{Lem=PsiBoundNew} and complex interpolation there is some absolute constant $C >0$ such that
 \[
 \Vert T_\rho^{A} :  \cS_{p_{k+1}} \rightarrow  \cS_{p_{k+1}} \Vert, \Vert T_\psi^{A} :  \cS_{p_{k+1}} \rightarrow  \cS_{p_{k+1}} \Vert \leq C \max \{ p_{k+1}, p_{k+1}' \}.
  \]
  This concludes the proof in case $0 < k < n-1$. The cases $k = 0, n-1$ can be proved completely analogously by estimating the expressions \eqref{Eqn=One} and \eqref{Eqn=Two}.
\end{proof}

\begin{dfn}
We define for $f \in \Cn$ with $f^{[n]}$ bounded and $1 \leq p_1, \ldots, p_n < \infty$ with $\sum_{l=1}^n \frac{1}{p_l} = 1$,
\[
\begin{split}
M_n(f, p_1, \ldots, p_n)  = & \sup_{A \in B(H)_{sa}} \Vert T^{(A, \ldots, A)}_{f^{[n]}}: \cS_{p_1} \times \ldots \times \cS_{p_n} \rightarrow \cS_{1, \infty}  \Vert, \\
M_n^+(f, p_1, \ldots, p_n)  = & \sup_{A \in B(H)_{sa}, A \geq 0} \Vert T^{(A, \ldots, A)}_{f^{[n]}}: \cS_{p_1} \times \ldots \times \cS_{p_n} \rightarrow \cS_{1, \infty}  \Vert, \\
M_n^-(f, p_1, \ldots, p_n)  = & \sup_{A \in B(H)_{sa}, A \leq 0} \Vert T^{(A, \ldots, A)}_{f^{[n]}}: \cS_{p_1} \times \ldots \times \cS_{p_n} \rightarrow \cS_{1, \infty}  \Vert.
\end{split}
\]
\end{dfn}

\begin{prop}\label{Prop=Ind}
Using the assumptions of Lemma \ref{Lem=Induction} there exists a constant $C (p_1, \ldots, p_n) >0$ such that
\[
\begin{split}
& M_n(f, p_1, \ldots, p_n) \\
  \leq  & 4 M_n^+(f, p_1, \ldots, p_n)  + 4 M_n^-(f, p_1, \ldots, p_n)  + C (p_1, \ldots, p_n) \cdot \\
& \cdot \left(     \max_{1 \leq k \leq n-1}  M_{n-1}(f_1, p_1, \ldots, p_{k-1}, \frac{p_k p_{k+1}}{p_k + p_{k+1}} ,  p_{k+2}, \ldots,  p_n)
  +   \Vert f^{[n]} \Vert_\infty \right).
\end{split}
\]
\end{prop}
\begin{proof}
Take $A \in B(H)_{sa}$ with $0 \not \in \sigma_p(A)$. By \cite[Remark 5.4]{PSS} we have that there exists a constant $C_0(p_1, \ldots, p_n) > 0$ such that
\[
\begin{split}
\Vert T^{(A, \ldots, A)}_{f^{[n-1]}_1} : \cS_{p_1} \times \ldots  \times \cS_{p_{n-1}} \rightarrow \cS_{q_1} \Vert  <& C_0(p_1, \ldots, p_n) \Vert f^{[n]} \Vert_\infty, \\
   \Vert T^{(A, \ldots, A)}_{f^{[n-1]}_1} : \cS_{p_2} \times \ldots  \times \cS_{p_n} \rightarrow \cS_{q_2} \Vert <& C_0(p_1, \ldots, p_n) \Vert f^{[n]} \Vert_\infty,
\end{split}
\]
where $\frac{1}{q_1} = \sum_{i=1}^{n-1} \frac{1}{p_i}, \frac{1}{q_2} = \sum_{i=2}^{n} \frac{1}{p_i}$.  Therefore by Lemma \ref{Lem=Induction}  for $\mathcal{A} \not = \emptyset$ and $\mathcal{A} \not = \{ 0, \ldots, n\}$ there exists a constant $C_1(p_1, \ldots, p_n) >0$ such that for $x_k \in \cS_{p_k} \cap \cS_2$,
\begin{equation}\label{Eqn=ANonEmpty}
\begin{split}
& \Vert T^{(A, \ldots, A) }_{f^{[n]}}(x_1^{\mathcal{A}}, \ldots, x_n^{\mathcal{A}}) \Vert_{\cS_{1, \infty}} \\
\leq &  C_1(p_1, \ldots, p_n)  \left(     \max_{1 \leq k \leq n-1}  M_{n-1}(f_1, p_1, \ldots, p_{k-1}, \frac{p_k p_{k+1}}{p_k + p_{k+1}} ,  p_{k+2}, \ldots,  p_n)  +   \Vert f^{[n]} \Vert_\infty \right)  \prod_{k=1}^n \Vert x_k \Vert_{\cS_{p_k}}.
\end{split}
\end{equation}
Further for $\mathcal{A} = \emptyset$   we have
\[
T^{(A, \ldots, A) }_{f^{[n]}}(x_1^{\mathcal{A}}, \ldots, x_n^{\mathcal{A}}) = T^{(A_{-}, \ldots, A_-) }_{f^{[n]}}(x_1^{\mathcal{A}}, \ldots, x_n^{\mathcal{A}}), \qquad A_- = \chi_{(- \infty, 0]}(A).
\]
And similarly  for $\mathcal{A} = \{0, \ldots, n \} $   we have
\[
T^{(A, \ldots, A) }_{f^{[n]}}(x_1^{\mathcal{A}}, \ldots, x_n^{\mathcal{A}}) = T^{(A_{+}, \ldots, A_+) }_{f^{[n]}}(x_1^{\mathcal{A}}, \ldots, x_n^{\mathcal{A}}), \qquad A_+ = \chi_{[0, \infty) }(A).
\]
It follows that
\begin{equation}\label{Eqn=AEmpty}
\Vert T^{(A, \ldots, A) }_{f^{[n]}}(x_1^{\mathcal{A}}, \ldots, x_n^{\mathcal{A}}) \Vert_{\cS_{1, \infty}} \leq M_n^{\pm } (f, p_1, \ldots, p_n)   \prod_{k=1}^n \Vert x_k \Vert_{\cS_{p_k}},
\end{equation}
where $\pm = -$ if $\mathcal{A} = \emptyset$   and $\pm = +$ if $\mathcal{A} = \{0, \ldots, n\}$.
Then we use  \eqref{Eqn=TADecomposition} and the quasi-triangle inequality followed by estimates  \eqref{Eqn=ANonEmpty}, \eqref{Eqn=AEmpty} to get
\begin{equation}\label{Eqn=LatterEstimate}
\begin{split}
& \Vert T^{(A, \ldots, A)}_{f^{[n]}}(x_1, \ldots, x_n) \Vert_{\cS_{1, \infty}} \\
\leq &
2^{n+1} \sum_{ \substack{ \mathcal{A} \subseteq \{ 1, \ldots, n\},\\  \mathcal{A} \not = \emptyset,  \mathcal{A}  \not = \{ 1, \ldots, n \} } } \Vert T^{(A, \ldots, A) }_{f^{[n]}}(x_1^{\mathcal{A}}, \ldots, x_n^{\mathcal{A}}) \Vert_{\cS_{1, \infty}}
+ 4 \sum_{\mathcal{A}   = \emptyset,   \{ 1, \ldots, n \} } \Vert T^{(A, \ldots, A) }_{f^{[n]}}(x_1^{\mathcal{A}}, \ldots, x_n^{\mathcal{A}}) \Vert_{\cS_{1, \infty}} \\
\leq & \big(   2^{2n+1}  C_1(p_1, \ldots, p_n)  \left(    \max_{1 \leq k \leq n-1}  M_{n-1}(f_1, p_1, \ldots, p_{k-1}, \frac{p_k p_{k+1}}{p_k + p_{k+1}} ,  p_{k+2}, \ldots,  p_n) +   \Vert f^{[n]} \Vert_\infty \right)   \\
&     + 4 M_n^+(f, p_1, \ldots, p_n)  + 4 M_n^-(f, p_1, \ldots, p_n) \big)  \prod_{k=1}^n \Vert x_k \Vert_{\cS_{p_k}} \\
\end{split}
\end{equation}
This estimate \eqref{Eqn=LatterEstimate} is uniform in $A \in B(H)_{sa}$ with $0 \not \in \sigma_p(A)$. Therefore by Proposition \ref{Prop=ApproxZero} the estimate \eqref{Eqn=LatterEstimate} holds uniformly for every $A \in B(H)_{sa}$ which is exactly the desired estimate.
\end{proof}

For $\pi_1, \pi_2$ disjoint subsets of $\mathbb{N}$ we write $\pi_1 < \pi_2$ if every element in $\pi_1$ is (strictly) smaller than every element in $\pi_2$.

\begin{dfn}
We say that $(q_1, \ldots, q_k)$ is a consummation of $(p_1, \ldots, p_n)$ if there exists a partition $\{ \pi_1, \ldots, \pi_k \}$ of $\{1, \ldots, n\}$ with $\pi_1 <\pi_2 <  \ldots < \pi_{n}$ such that
\[
\frac{1}{q_k} = \sum_{l \in \pi_k} \frac{1}{p_l}.
\]
\end{dfn}

For $f \in \Cn$ with $f^{[n]}$ bounded and $1 \leq p_1, \ldots, p_n < \infty$ with $\sum_{l=1}^n \frac{1}{p_l}$ set
\begin{equation}
\begin{split}
L_n^+(f, p_1, \ldots, p_n) = & \sup \{  M_{n-k}^+(f_k, q_1, \ldots, q_k) \mid 0 \leq k < n  \}, \\
L_n^-(f, p_1, \ldots, p_n) = & \sup \{  M_{n-k}^-(f_k, q_1, \ldots, q_k) \mid 0 \leq k < n  \},
\end{split}
\end{equation}
where the suprema are taken over all consummations $(q_1, \ldots, q_k)$ of $(p_1, \ldots, p_n)$.

\begin{thm}\label{Thm=Ultimate}
Let $f \in \Cn$ with $f^{[n]}$ bounded and let  $1 \leq p_1, \ldots, p_n < \infty$ with $\sum_{l=1}^n \frac{1}{p_l} = 1$. There exists a constant $C(p_1, \ldots, p_n)>0$ such that for every $\boldA \in B(H)_{sa}^{\times n+1}$ we have
\begin{equation}\label{Eqn=MainEstimate}
\Vert T_{f^{[n]}}^{\boldA} : \cS_{p_1} \times \ldots \times \cS_{p_n} \Vert \leq C(p_1, \ldots, p_n) \left(
\Vert f^{[n]} \Vert_\infty + L_n^+(f, p_1, \ldots, p_n) + L_n^-(f, p_1, \ldots, p_n) \right).
\end{equation}
\end{thm}
\begin{proof}
Proposition \ref{Prop=Ind} shows that  for every $0 \leq k < n-1$ and  any consummation $(q_1, \ldots, q_{n-k})$ of  $(p_1, \ldots, p_{n})$  there exists a constant $C(q_1, \ldots, q_{n-k})$ such that
\begin{equation}\label{Eqn=IndEstimate}
\begin{split}
& M_{n-k}(f_k, q_1, \ldots, q_{n-k}) \\
  \leq  & 4 L_n^+(f, p_1, \ldots, p_n)  + 4 L_n^-(f, p_1, \ldots, p_n) + C(q_1, \ldots, q_{n-k}) \cdot \\
& \cdot \left(     \max_{1 \leq k \leq n-1}  M_{n-k-1}( f_{k+1} , q_1, \ldots, q_{k-1}, \frac{q_k q_{k+1}}{q_k + q_{k+1}} ,  q_{k+2}, \ldots,  q_n)
  +   \Vert f^{[n]} \Vert_\infty \right).
\end{split}
\end{equation}
If $k = n-1$ then by \cite[Theorem 1.2]{CPSZ-AJM} there exists $C>0$ such that
\begin{equation}\label{Eqn=IndFinal}
M_{1}(f_{n-1}, 1) \leq C \Vert f^{[n]} \Vert_\infty.
\end{equation}
Applying the estimate \eqref{Eqn=IndEstimate}  inductively from $k=0$ to $k = n-2$ and using \eqref{Eqn=IndFinal} for $k = n-1$ we see that there is a constant $C(p_1, \ldots, p_n) >0$ such that
\[
M_{n}(f, p_1, \ldots, p_{n} ) \leq C(p_1, \ldots, p_n) ( L_n^+(f, p_1, \ldots, p_n)  +  L_n^-(f, p_1, \ldots, p_n) + \Vert f^{[n]} \Vert_\infty).
\]
This is the desired estimate  \eqref{Eqn=MainEstimate} for $\boldA = (A, \ldots, A), A \in B(H)_{sa}$. The general case follows from Corollary \ref{Cor=Amplify}.
\end{proof}

\section{Consequences of Theorem \ref{Thm=Ultimate}: Weak $(1,1)$ estimates for generalized absolute value functions} \label{Sect=Consequences}

We now arrive at the applications of Theorem \ref{Thm=Ultimate}.

 \begin{thm}\label{Thm=AbsValue}
Let $a(t) = \vert t \vert t^{n-1}, t \in \mathbb{R}$. Fix $1<p_1, \ldots, p_n< \infty$ such that $\sum_{l=1}^n \frac{1}{p_l} = 1$. Then there exists a constant $D >0$ such that for  every $\boldA \in B(H)_{sa}^{\times n +1}$ we have that
\begin{equation}
\Vert T^{\boldA}_{a^{[n]}}: \cS_{p_1} \times \ldots \times \cS_{p_n} \rightarrow \cS_{1, \infty}\Vert \leq D.
\end{equation}
\end{thm}
\begin{proof}
Let $\epsilon = \pm 1$. Set $b(t) = t^n$. Then $a(t) = \epsilon b(t)$ for every $t \in \mathbb{R}$ with $\epsilon t \geq 0$. Consequently, $a^{[n]}(t) = \epsilon b^{[n]}(t)$ for every $t \in \epsilon \cdot \mathbb{R}^{n+1}_{\geq 0}$. Recall that $a_k$ and $b_k$ are defined in \eqref{Eqn=Fl} with $f$ replaced by $a$ and $b$ respectively.
 So  by Lemma \ref{Lem=PostCrucial} we certainly have
\begin{equation}\label{Eqn=AisB}
a_k^{[n-k]}(t) = \epsilon b_k^{[n-k]}(t), \qquad    t \in \epsilon \cdot \mathbb{R}^{n-k+1}_{\geq 0}, 0 \leq k <  n.
\end{equation}
 Further, for the $n$-th order derivative we have $b^{(n)}(t) =  n!$ for $t \in \mathbb{R}$. By the integral expression for divided differences \cite[Lemma 5.1]{PSS} we conclude that
 $b^{[n]}(t) =  n!$ for $t \in \mathbb{R}^{n+1} \backslash \{ 0 \}$ and so by \eqref{Eqn=AisB} we find  $a^{[n]}(t) = \epsilon n!$ for all $t \in \epsilon \cdot \mathbb{R}^{n+1}_{\geq 0} \backslash \{0 \}$.  So by Lemma \ref{Lem=PostCrucial} we certainly have,
 \[
 a_k^{[n-k]}(t) = \epsilon n!, \qquad  t \in  \epsilon \cdot  \mathbb{R}^{n-k+1}_{\geq 0} \backslash \{0 \} , 0 \leq k <  n.
 \]
 Let $\chi_{\epsilon, k}$ be the indicator function of $\epsilon \cdot \mathbb{R}^{n-k+1}_{\geq 0}$.
   Then $\chi_{\epsilon, k} a^{[n-k]}_k = \epsilon n! ( \chi_{\epsilon, k }   - \chi_{ 0 } )$.  We conclude from Lemma \ref{Lem=Indicator} that for $\epsilon A \geq 0$,
 \[
  T^{(A, \ldots, A)}_{    a_k^{[n-k]}} =
 T^{(A, \ldots, A)}_{  \chi_{\epsilon, k  }   a_k^{[n-k]}} =
\epsilon n! ( T^{(A, \ldots, A) }_{ \chi_{ \epsilon, k   }  } -  T^{ (A, \ldots, A) }_{ \chi_{ 0 } }).
 \]
 By Lemma \ref{Lem=Indicator}
 the  multiple operator integrals  $T^{(A, \ldots, A) }_{ \chi_{ \epsilon  , k   }  }$ and $T^{ (A, \ldots, A) }_{ \chi_{ 0 } }$ are contractions $\cS_{r_1} \times \ldots \times \cS_{r_{n-k}} \rightarrow \cS_{1}$ for any   $(r_1, \ldots, r_{n-k})$ with $\sum_{l=1}^{n-k} \frac{1}{r_l} = 1$. So certainly they are contractions $\cS_{r_1} \times \ldots \times \cS_{r_{n-k}} \rightarrow \cS_{1, \infty}$.  We conclude that
 \[
\begin{split}
L_n^{\epsilon}(a, p_1, \ldots, p_n) < 2 n!.
\end{split}
\]
 Hence we conclude the theorem from  Theorem \ref{Thm=Ultimate}.
 \end{proof}

 \begin{thm}\label{Thm=AbsValuePlus}
Let $g \in C^{n+1}(\mathbb{R})$ be such that $g(t) = \vert t \vert t^{n-1}, t \in \mathbb{R} \backslash [-1,1]$. Fix $1<p_1, \ldots, p_n< \infty$ such that $\sum_{l=1}^n \frac{1}{p_l} = 1$. Then there exists a constant $D >0$ such that for  every $\boldA \in B(H)_{sa}^{\times n +1}$ we have that
\begin{equation}
\Vert T^{\boldA}_{g^{[n]}}: \cS_{p_1} \times \ldots \times \cS_{p_n} \rightarrow \cS_{1, \infty}\Vert \leq D.
\end{equation}
\end{thm}
\begin{proof}
Set $c = g - a$  where $a$ is defined in Theorem \ref{Thm=AbsValue}.
From the definition of divided differences $c^{[n]} = g^{[n]} - a^{[n]}$. By Theorem \ref{Thm=AbsValue} and the quasi-triangle inequality it  suffices to show that $T_{c^{[n]}}^{\boldA}$ is bounded $\cS_{p_1} \times \ldots \times \cS_{p_n} \rightarrow \cS_{1, \infty}$ uniformly in $\boldA \in B(H)^{\times n+1}_{sa}$.

Note that by the assumptions $c$ is supported on $[-1, 1]$.  Let $\epsilon = \pm 1$. Let $B: \mathbb{R} \rightarrow [0,1]$ be a smooth compactly supported function that is 1 on the interval $\epsilon [0,1]$. Set
\[
c_\epsilon(t) = B(t)(g(t) - \epsilon t^n), \qquad t \in \mathbb{R}.
\]
 Then $c_\epsilon$ is a compactly supported $C^{n+1}$-function and $c_\epsilon(t) = c(t), t\in \epsilon \mathbb{R}_{\geq 0}$.   Therefore
 for all  $t \in \epsilon \cdot  \mathbb{R}_{\geq 0}^{n+1}$ we have
 $c_\epsilon^{[n]}(t) = c^{[n]}(t)$. By Lemma \ref{Lem=PostCrucial}  for all  $t \in \epsilon \cdot  \mathbb{R}_{\geq 0}^{n-k+1}$ we have
 $c_{\epsilon,k}^{[n-k]}(t) = c^{[n-k]}_k(t)$.
  In other words, if  $\chi_{\epsilon,k}$ is the indicator function on $\epsilon \cdot \mathbb{R}_{\geq 0}^{n-k+1}$,
 \begin{equation}\label{Eqn=DomainEq}
c_{\epsilon,k}^{[n-k]} \chi_{\epsilon, k}  = c^{[n-k]}_k \chi_{\epsilon,k}.
 \end{equation}
 By Proposition  \ref{Prop=EasyClass} there  exists a constant $D>0$ such that for all $A \in B(H)_{sa}$ we have
\begin{equation}\label{Eqn=Tbep}
\Vert T_{c_{\epsilon,k}^{[n-k]}}^{(A, \ldots, A)}: \cS_{p_1} \times \ldots \times \cS_{p_n} \rightarrow \cS_1 \Vert \leq D.
\end{equation}
We see by \eqref{Eqn=DomainEq},  \eqref{Eqn=Tbep} and Lemma \ref{Lem=Indicator} that for all $A \in B(H)_{sa}$ with $\epsilon A \geq 0$,
\[
\Vert  T_{c^{[n-k]  }_k  }^{(A, \ldots, A) } =  T_{c^{[n-k]  }_k \chi_{\epsilon,k} }^{(A, \ldots, A) }   =  T_{  c_{\epsilon,k}^{[n-k]}  \chi_{\epsilon,k} }^{(A, \ldots, A)}: \cS_{p_1} \times \ldots \times \cS_{p_n} \rightarrow \cS_1 \Vert \leq D.
\]
 So in Theorem \ref{Thm=Ultimate}
we have that $L_n^{\epsilon}(c, p_1, \ldots, p_n) < \infty$   and so by the same Theorem \ref{Thm=Ultimate} we conclude the proof.
\end{proof}

\begin{rmk}\label{Rmk=Final}
In \cite[Lemma 28]{CLPST}, see in particular the line after equation (37) of \cite{CLPST},  the following result was proved and is a key step in the resolution of Peller's problem as stated in \cite{CLPST}. Let $n=2$ and let  $g: \mathbb{R} \rightarrow \mathbb{C}$ be a function as in the statement of Theorem \ref{Thm=AbsValuePlus} with the additional assumption that $g(0) = g'(0) = g''(0) =0$. There exists  no constant $0 < D < \infty$ such that for all $\boldA \in B(H)_{sa}^{\times 3}$ we have
\[
\Vert T_{g}^{\boldA}: \cS_{2} \times \cS_2 \rightarrow \cS_1 \Vert \leq D.
\]
  This shows that Theorem \ref{Thm=AbsValuePlus} is optimal.
\end{rmk}

{\bf Acknowledgement:} the authors are grateful to the referee for detailed comments which helped to improve the exposition.

\end{document}